\documentclass[reqno,12pt]{amsart}

\usepackage[utf8]{inputenc}
\usepackage[T1]{fontenc}

\usepackage{fourier}
\usepackage{amssymb , amsmath, amsthm}

\usepackage{hyperref}

\newtheorem{defi}{Definition}
\newtheorem{theorem}[defi]{Theorem}
\newtheorem{ex}[defi]{Example}

\newtheorem{rem}[defi]{Remark}
\newtheorem{prop}[defi]{Proposition}
\newtheorem{lemma}[defi]{Lemma}
\newtheorem{cor}[defi]{Corollary}

\newcommand{\R}{\mathbb{R}}
\newcommand{\N}{\mathbb{N}}
\newcommand{\NM}{N_M}
\newcommand{\NS}{N_\Sigma}
\newcommand{\G}{\Gamma}

\newcommand{\eps}{\varepsilon}
\newcommand{\Sp}{\mathbb{S}}

\newcommand{\nd}{\partial_\nu}

\newcommand\vol[1]{\vert#1\vert}
\newcommand\volM[1]{\vert#1\vert_M}
\newcommand\volS[1]{\vert#1\vert_\Sigma}

\newcommand{\dtn}{\mathcal{D}}

\usepackage{setspace}
\usepackage{marginnote}
\let\oldmarginnote\marginnote
\renewcommand{\marginnote}[1]
{\oldmarginnote{
\vspace{-.3cm}\red{\tiny
  \begin{spacing}{1}#1
\end{spacing}}}
}

\title[Metric upper bounds]{Metric upper bounds for Steklov and Laplace eigenvalues}

\author{Bruno Colbois}
\address{Universit\'e de Neuch\^atel, Institut de Math\'ematiques, Rue
  Emile-Argand 11, CH-2000 Neuch\^atel, Switzerland}
\email{bruno.colbois@unine.ch}

\author{Alexandre Girouard}
\address{D\'epartement de math\'ematiques et de statistique, Pavillon Alexandre-Vachon, Universit\'e Laval, Qu\'ebec, QC, G1V 0A6, Canada}
\email{alexandre.girouard@mat.ulaval.ca}

\date{\today}

\begin{document}
\begin{abstract}
  We prove two upper bounds for the Steklov eigenvalues of a compact Riemannian manifold with boundary. The first involves the volume of the manifold and of its boundary, as well as packing and volume growth constants of the boundary and its distortion. Its proof is based on a metric-measure space technique that was introduced by Colbois and Maerten. The second bound is in terms of the extrinsic diameter of the boundary and its injectivity radius. It is obtained from a concentration inequality, akin to Gromov--Milman concentration for closed manifolds. By applying these bounds to cylinders over closed manifold, we obtain bounds for eigenvalues of the Laplace operator, in the spirit of Grigor'yan--Netrusov--Yau and of Berger--Croke. For a family of manifolds that has uniformly bounded volume and boundary of fixed intrinsic geometry, we deduce that a large first nonzero Steklov eigenvalue implies that each boundary component is contained in a ball of small extrinsic radius.
\end{abstract}
\maketitle

\section{\bf Introduction}
Let $M$ be  a smooth connected compact Riemannian manifold of dimension $n+1\geq 2$, with boundary $\Sigma=\partial M$.
The Dirichlet-to-Neumann operator $\dtn:C^\infty(\Sigma)\to C^\infty(\Sigma)$ is defined by
$\dtn f=\nd{\hat{f}}$, where $\nu$ is the outward normal  along the boundary $\Sigma$ and where the function $\hat{f}\in C^\infty(M)$ is the unique harmonic extension of $f$ to the interior of $M$. 
The eigenvalues of $\dtn$ are known as \emph{Steklov eigenvalues} of $M$. They form an unbounded sequence
$0=\sigma_0\leq\sigma_1\leq\sigma_2\leq\cdots\to\infty,$
where as usual each eigenvalue is repeated according to its multiplicity.
The interplay of these eigenvalues with the geometry of $M$ has been an active area of investigation in recent years. See~\cite{GPsurvey,CGGSsurvey} for surveys and~\cite{CGG,ColboisGittins,ColboisVerma2020,Kokarev2021,Brisson2022,GL,GKL} for recent relevant results.

In this paper we study upper bounds for Steklov eigenvalues in terms of geometric quantities that are metric in nature: packing and growth constants, distortion between the intrinsic and extrinsic distances on the boundary, as well as diameters and injectivity radius of the boundary components.
A recuring feature is that the bounds are linked to some comparison between intrinsic and extrinsic geometry of the boundary. They do not involve the curvature. See~\cite{Korevaar, GNY, CM} for early use of similar techniques and \cite{CEG1, CEG2, CGG} some more recent results in the same spirit. See~\cite{CGH,Xi,IliasMakhoul2011} for some bounds depending on curvature asumptions.
Upper bounds for the eigenvalues $\lambda_k$ of the Laplacian on a closed Riemannian manifold will also be obtained. They are in the spirit of~\cite{Berger1979, Croke1980, Kokarev2021}, and of \cite{GNY, CM}.

\subsubsection*{Notations}
We use two distances on the boundary $\Sigma$. The first one is the geodesic distance $d_\Sigma$. The second distance is induced on $\Sigma$ from the geodesic distance $d_M$ in $M$. In general, the letters $M$ and $\Sigma$ will be used to specify which distance is involved. For instance,  for $x\in M$ we define the ball
$$B^M(x,r):=\{y\in M\,:\,d_M(x,y)<r\},$$
and for $x\in\Sigma$,
$$B^\Sigma(x,r):=\{y\in \Sigma\,:\,d_\Sigma(x,y)<r\}.$$
Similarly, we write $\vol{\mathcal{O}}_M$ for the Riemannian measure of a Borel set $\mathcal{O}\subset M$, while  $\vol{\mathcal{O}}_\Sigma$ is the Riemannian measure of $\mathcal{O}\cap\Sigma$.

\subsection{Upper bound in term of metric invariants}
For $x,y\in\Sigma$, we have $d_M(x,y)\leq d_\Sigma(x,y)$, where the convention is that for $x$ and $y$ in different connected components of $\Sigma$, we set $d_\Sigma(x,y)=+\infty$. 
Let $\Sigma_1,\cdots,\Sigma_b$ be the connected components of the boundary $\Sigma$.
The \emph{distortion of\, $\Sigma_j$ in $M$} is the number $\Lambda_j\in [1,\infty)$ defined by 
\begin{equation}\label{eq:condition}
\Lambda_j:=\inf\{c\geq 1\,:\ d_\Sigma(x,y)\le cd_M(x,y)\quad\forall x,y\in\Sigma_j\}.
\end{equation}
The distortion of $\Sigma$ in $M$ is
$$\Lambda:=\max\{\Lambda_1,\cdots,\Lambda_b\}.$$
The distortion is a measure of how much the geodesic distance $d_\Sigma$ differs from the induced distance $d_M\bigl\vert_\Sigma\bigr.$.
To state our first main result we also need two more geometric invariants.
It follows from the compactness of\, $\Sigma$ that there exist a \emph{packing constant} $\NM\in\N$ for $(\Sigma,d_M)$ and a \emph{packing constant} $N_\Sigma$ for $(\Sigma,d_\Sigma)$, which satisfy the following properties:
\begin{itemize}
\item  For each $r>0$ and each $x\in\Sigma$, the extrinsic ball $B^M(x,r)\cap\Sigma$ can be covered by $\NM$ extrinsic balls of radius $r/2$ centered at points $x_1,\cdots,x_{\NM}\in\Sigma$:
  $$B^M(x,r)\cap\Sigma\subset \bigcup_{i=1}^{\NM}B^M(x_i,r/2);$$
\item For each $r>0$ and each $x\in\Sigma$, the intrinsic ball $B^\Sigma(x,r)$ can be covered by $\NS$ intrinsic balls of radius $r/2$ centered at points\\ $x_1,\cdots,x_{\NS}\in\Sigma$:
  $$B^\Sigma(x,r)\subset \bigcup_{i=1}^{\NS}B^\Sigma(x_i,r/2);$$
\end{itemize}
There also exists a \emph{growth constant }$\G>1$ for the metric-measure space $(\Sigma, d_\Sigma,\vol{\cdot}_\Sigma)$, which satisfies the following property: for each $x\in\Sigma$ and each $r>0$, $|B^\Sigma(x,r)|_\Sigma\leq \G r^n$.

The first main result of this paper is a generalization of~\cite[Theorem 1.1]{CGG}, where $M$ was a submanifold in some euclidean space.
\begin{theorem}\label{thm:upperbound}
  Let $M$ be a smooth connected compact Riemannian manifold of dimension $n+1$ with boundary $\Sigma$.
  The following holds for each $k\ge 1,$
  \begin{gather}\label{ineq:upperboundintro}
    \sigma_k\le 512b^2\NM^3\G\Lambda^2\frac{\vert M\vert}{\vert \Sigma\vert^{\frac{n+2}{n}}}k^{2/n}.
  \end{gather}
The exponent $2/n$ on $k$ is optimal: it cannot be replaced by any smaller number.
\end{theorem}
The proof of Theorem~\ref{thm:upperbound} will be presented in Section~\ref{section:upperbounds}, where it will be deduced from a slightly more general statement (Theorem~\ref{thm:upperboundGeneral}). The optimality of the exponent is discussed in Remark~\ref{remark:optimal}.

The  volume $|M|$, the distortion $\Lambda$ and the packing constant $\NM$ depend on the geometry of $M$ in its interior, while  the constants $b,\G, |\Sigma|$ only depend  on the intrinsic geometry of the boundary $\Sigma$. In fact, the extrinsic packing consant $N_M$ can be expressed in terms of the distortion and of the intrinsic packing constant $N_\Sigma$ of $(\Sigma, d_\Sigma)$. This will be proved in Lemma~\ref{lemma:packing}. This leads to the following
  \begin{cor}\label{cor:upperboundmodif}\
    Under the hypothesis of Theorem~\ref{thm:upperbound}, the following holds for each $k\geq 1,$
  \begin{gather}\label{ineq:upperboundmodif}
    \sigma_k\le 512b^5\NS^{3\log_2(2\Lambda)}\G \Lambda^2\frac{\vert M\vert}{\vert \Sigma\vert^{\frac{n+2}{n}}}k^{2/n}.
\end{gather}
\end{cor}
  In inequality~\eqref{ineq:upperboundmodif}, apart from $\sigma_k$, only the distortion $\Lambda$ and the volume of $M$ depend on the geometry of $M$. All other geometric quantities are intrinsic to the boundary $\Sigma$. The importance of each geometric constant appearing in~\eqref{ineq:upperboundmodif} will be discussed in section~\ref{subsection:relevance}.

While inequality~\eqref{ineq:upperboundmodif} is somewhat cumbersome, its strenght is that its geometric dependance is completely explicit, with clear distinction between extrinsic and intrinsic features. The reader is invited to compare with~\cite[Theorem 1.1]{CGG}. Note also that none the geometric invariants appearing in inequality~\eqref{ineq:upperboundintro} is superfluous. This will be discussed in Section~\ref{subsection:relevance}.

\begin{rem}
  One can rewrite~\eqref{ineq:upperboundintro} in the following scale-invariant fashion:
\begin{gather}\label{ineq:reformulation}
    \sigma_k|\Sigma|^{1/n}\le \frac{512}{I(M)^{\frac{n+1}{n}}}b^2\NM^3\G \Lambda^2k^{2/n},
\end{gather}
  where $I(M)=|\Sigma|/|M|^{\frac{n}{n+1}}$ is the isoperimetric ratio of $M$. One should compare this with~\cite[Theorem 1.3]{CEG1}, which states that
  \begin{gather}\label{ineq:CEGsteklov}
    \sigma_k|\Sigma|^{1/n}\le \frac{\gamma(n)}{I(M)^{\frac{n-1}{n}}}k^{2/(n+1)},
  \end{gather}
  for domains $M$ in a complete space that is conformally equivalent to a complete manifold with non-negative Ricci curvature. Our new inequality~\eqref{ineq:reformulation} applies to a much larger class of manifolds, since no curvature asumption is required.
\end{rem}

There is a close link between the Steklov eigenvalues of a manifolds $M$ and the Laplace eigenvalues of its boundary $\partial M$. See~\cite{CGH}. In the situation where $M$ is a cylinder this link is explicit and can be used to obtain new bounds on Laplace eigenvalues from known results on Steklov eigenvalues.
Given a compact connected Riemannian manifold $\Sigma$, the eigenvalues of the Laplace operator $\Delta:C^\infty(\Sigma)\longrightarrow C^\infty(\Sigma)$ are written $0=\lambda_0<\lambda_1\leq\lambda_2\leq\cdots\to\infty$.
\begin{cor}\label{coro:gnycm}
  Let $\Sigma$ be a compact connected Riemannian manifold of dimension $n$. Then,
    \begin{gather}\label{ineq:lambdak}
    \lambda_k\vert\Sigma\vert^{2/n}\le 2048\G\NS^3k^{2/n},
    \end{gather}
    where $\NS$ a packing constant of $(\Sigma,d_\Sigma)$.
  \end{cor}
  
  This result is similar in spirit to those presented in~\cite{GNY} and~\cite{CM}. This is not surprising since the proof of Theorem~\ref{thm:upperbound} is based on a simplification of the main technical tool from~\cite{CM}, presented here as Lemma~\ref{CMrevisited}. One should compare this with~\cite[Remark 5.10]{GNY}.
\begin{proof}[Proof of Corollary \ref{coro:gnycm}]
  The Steklov eigenvalues of the cylinder $M=[0,L]\times\Sigma$ have been computed in~\cite[Lemma 6.1]{CEG1}. For $L>0$ small enough, 
  $\sigma_k=\sqrt{\lambda_k}\tanh(\sqrt{\lambda_k}L)$, and it follows from inequality~\eqref{ineq:upperboundintro} that
  \begin{gather*}
    \sqrt{\lambda_k}\tanh(\sqrt{\lambda_k}L)\le 2048N^3\G \Lambda^2\frac{L\vert \Sigma\vert}{\vert \Sigma\vert^{\frac{n+2}{n}}}k^{2/n}.
  \end{gather*}
    Dividing by $L$ on each side and taking the limit as  $L\to 0$ completes the proof, since for each $c>0$ the following holds:
  $$\lim_{x\to 0}\frac{c\tanh(cx)}{x}=c^2.$$
\end{proof}
\begin{rem}\label{remark:optimal}
  It follows from Weyl's law for $\lambda_k$ that the power ${2/n}$ on $k$ is optimal in~\eqref{ineq:lambdak}, and therefore also in~\eqref{ineq:upperboundintro}.
\end{rem}

\subsection{A Berger-Croke type inequality for  Steklov and Laplace eigenvalues}

Let $M$ be a compact manifold with boundary $\Sigma$. Let $\Sigma_1,\cdots,\Sigma_b$ be the connected components of $\Sigma$. The diameter of a connected component $\Sigma_j$ is defined by
$$\text{Diam}_M(\Sigma_j):=\sup\left\{d_M(x,y)\,:\,x,y\in\Sigma_j\right\}.$$
The main result of this section is the following.
\begin{theorem}\label{thm:upperboundDiam}
    Let $\Sigma_j$ be a connected component of the boundary $\Sigma$. Then,
    \begin{gather}\label{ineq:upperboundDiamNew}
      \sigma_k\leq K(n)\frac{|M|}{\text{diam}_M(\Sigma_j)^2}\times\frac{1}{\min\{\text{diam}_M(\Sigma_j)^n,\, \text{inj}(\Sigma_j)^n\}}k^{n+1},
    \end{gather}
    where $K(n)$ is a  dimensional constant and $\text{inj}(\Sigma_j)$ is the injectivity radius of\, $\Sigma_j$.
  \end{theorem}

\begin{rem}
  The diameter of $M$ itself does not appear in inequality~\eqref{ineq:upperboundDiamNew}. This is not surprising, since one can modify $M$ away from the boundary so as to obtain arbitrarily large diameter
  $$\sup\{d_M(x,y)\,:\,x,y\in M\},$$
  without significant change to $\sigma_k$, the extrinsic diameter $\text{Diam}_M(\Sigma)$ and the volume $\vert M\vert$. This can be performed for instance by replacing two small balls in $M$ with a long thin tube joining them. See~\cite[Theorem 1.2]{FS4}.
\end{rem}
\begin{rem}
  The presence of the injectivity radius in the denominator of inequality~\eqref{ineq:upperboundDiamNew} is essential. Indeed, let $(M,g)$ be a compact Riemannian manifold of dimension $n\geq 3$ and let $C\subset M$ be a smooth embedded closed curve. For $\eps>0$ small enough
  $$\Omega_\eps:=\{x\in M\,:\,d_g(x,C)>\eps\}$$
  is a connected domain with smooth boundary. In her paper~\cite{Brisson2022}, Jade Brisson proved that
  $\sigma_1(\Omega_\eps)\xrightarrow{\eps\to 0}\infty.$
  Because $|\Omega_\eps|$ and  $\text{Diam}_M(\partial\Omega_\eps)$ are uniformly bounded as $\eps\to 0$, the injectivity radius could not be removed from inequality~\eqref{ineq:upperboundDiamNew}.
\end{rem}

For a family of manifolds that has uniformly bounded volume and boundary of fixed intrinsic geometry, we deduce that a large Steklov eigenvalue $\sigma_k$ (for a fixed index $k$) implies that each boundary component is concentrated in a ball of small extrinsic radius.
\begin{cor}\label{coro:concentrationNEW}
    Let $(M_\eps)_{\eps>0}$ be a family of compact Riemannian manifolds of dimension $m$, with boundary $\Sigma_\eps$. Suppose that $|M_\eps|<\alpha$ and $\text{inj}(\Sigma_\eps)>\beta$ for some positive constants. If $\sigma_k(M_\eps)\xrightarrow{\eps\to 0}+\infty$, then the connected components $\Sigma_{j,\eps}$ of the boundary satisfie
    $$\lim_{\eps\to 0}\text{diam}_{M_\eps}(\Sigma_{j,\eps})=0.$$
\end{cor}
In the situation when it is the first nonzero eigenvalue $\sigma_1$ that is large, there is also a global concentration phenomenon. See Proposition~\ref{prop:concentrationNBD}.
The proof of Theorem~\ref{thm:upperboundDiam} itself is based on a simple concentration bound which is adapted from the work of Gromov and Milman~\cite{GM}. See below in Lemma~\ref{lemma:GMnew}.
\begin{rem}
The control of the injectivity radius is essential to obtain the type of concentration portrayed in Corollary~\ref{coro:concentrationNEW}. Indeed, consider the following construction:  let $M$ to be a compact manifold of dimension $\geq 3$ with connected boundary. Suppose that the volume of $M$ is one, while  $\sigma_1$ is very large. Let $p$ and $q$ be two points of the boundary and let $\gamma$ be a smooth curve in the interior of $M$ joining $p$ and $q$ and meeting the boundary of $M$ orthogonally. Fraser and Schoen have proved in~\cite{FS4} that removing a thin tubular neighborhood of $\gamma$ has neglictible effect on $\sigma_k$. Since the curve $\gamma$ could go very far inside $M$, the extrinsic diameter of the boundary could become very large without affecting $\sigma_1$, the volume of $M$ or the volume of the boundary significantly. This illustrates the necessity of controlling the injectivity radius. 
In fact, some form of concentration remains when the injectivity radius is not controlled, but it is a concentration of the measure, rather than a metric concentration. See Proposition~\ref{prop:concentrationNBD}.
\end{rem}

\begin{rem}
  Corollary~\ref{coro:concentrationNEW} should be compared with~\cite[Theorem 3]{CS2}, which essentially implies the following: if the first nonzero eigenvalue $\lambda_1$ of the Laplacian on a closed Riemannian manifold $M$  is large with respect to its packing constant, then the Riemannian measure of $M$ is concentrated near one point.
\end{rem}


\subsubsection*{Berger--Croke type inequality for eigenvalues of the Laplace operator}
Berger proved in~\cite{Berger1979} that on any closed Riemanian manifold $\Sigma$ which admits an isometric involution without fixed points, the first nonzero eigenvalue of the Laplacian satisfies
\begin{gather}\label{ineq:Berger1970}
  \lambda_1\leq K(n)\frac{|\Sigma|}{\text{inj}(\Sigma)^{n+2}}.
\end{gather}
This result was generalized by Croke, who proved in~\cite{Croke1980} that any closed Riemannian manifold satisfies the following for each $k\in\N$,
\begin{gather}\label{ineq:Croke1980}
  \lambda_k\leq K(n)\frac{|\Sigma|^2}{\text{conv}(\Sigma)^{2n+2}}k^{2n}.
\end{gather}
Here $\text{conv}(\Sigma)$ is the \emph{convexity radius} of $\Sigma$. 
One should also see the recent paper~\cite{Kokarev2021} by Kokarev.
Theorem~\ref{thm:upperboundDiam} leads to the following improvement of the Berger and Croke inequalities.
\begin{cor}\label{cor:bergerlambda}
    Let $\Sigma$ be a closed Riemannian manifold. Then for each $k\in\N$,
    \begin{gather}\label{ineq:BergerCroke}
      \lambda_k\text{diam}(\Sigma)^2\leq K(n)\frac{|\Sigma|}{\text{inj}(\Sigma)^n}k^{n+1}.
    \end{gather}
\end{cor}
  Inequality~\eqref{ineq:BergerCroke} improves on Berger and Croke in several ways. For instance, the exponent on $k$ is better. Perhaps more interestingly, because $\text{conv}(\Sigma)\leq\text{inj}(\Sigma)\leq\text{Diam}(\Sigma)$, inequality~\eqref{ineq:BergerCroke} allows a finer control for manifolds that have small injectivity radius and large radius, as the following example shows.
\begin{ex}
    The first nonzero Laplace eigenvalue of $\Sigma_L=\Sp^1_L\times \Sp^{n-1}$
    behaves as $\lambda_1\sim 1/L^2$ as $L\to+\infty$. Moreover, the volume of $\Sigma_L$ is $|\Sigma_L|=L|\Sp^{n-1}|$, its injectivity and convexity radii both are the constant $\pi$. If $L>0$ is large enough, the diameter of $\Sigma$ is of order $L$. Whence, the upper bound from Berger and Croke read
    $$\lambda_1\leq K(n)L|\Sp^{n-1}|/\pi^{n+2}\qquad\text{ and }\qquad \lambda_1\leq K(n)L^2|\Sp^{n-1}|^2/\pi^{2n+2}.$$
    Both upper bounds diverges as $L\to+\infty$.
    On the other hand, our bound~\eqref{ineq:BergerCroke} gives the much more accurate
    $$\lambda_1\leq K(n)\frac{|\Sp^{n-1}|}{L\pi^{n}}\xrightarrow{L\to\infty}0.$$
  \end{ex}  
Let us show how Corollary~\ref{cor:bergerlambda} is obtained from Theorem~\ref{thm:upperboundDiam}.
\begin{proof}[Proof of Corollary~\ref{cor:bergerlambda}]
  As in the proof of Corollary~\eqref{coro:gnycm}, we consider the Steklov eigenvalues of the cylinder $M=[0,L]\times\Sigma$. Notice that
  $$\text{diam}_M(\Sigma\times\{0\})=\text{diam}(\Sigma)$$ and $\text{inj}(\Sigma)\leq\text{diam}(\Sigma)$.
  For $L>0$ small enough,
$\sigma_k=\sqrt{\lambda_k}\tanh(\sqrt{\lambda_k}L)$ and it follows from inequality~\eqref{ineq:upperboundDiamNew} that
\begin{align*}
  \sqrt{\lambda_k}\tanh(\sqrt{\lambda_k}L)&\leq
  K(n)\frac{|M|}{\text{diam}(\Sigma)^2}\times\frac{1}{\min\{\text{diam}(\Sigma)^n,\, \text{inj}(\Sigma)^n\}}k^{n+1}\\
                                          &=
                                              K(n)\frac{|\Sigma|L}{\text{Diam}(\Sigma)^{2}\text{inj}(\Sigma)^n}k^{n+1}.
  \end{align*}
  Dividing by $L$ on each side and taking the limit as  $L\to 0$ completes the proof.
\end{proof}

\subsection*{Plan of the paper}
In Section~\ref{section:GromovMilmanDiam}, we present a concentration inequality akin to that of Gromov--Milman~\cite[Theorem 4.1]{GM}, which we use to prove Theorem~\ref{thm:upperboundDiam}.
To prove Theorem~\ref{thm:upperbound} some tools from metric geometry will then be used in Section~\ref{section:upperbounds}. In particular, Lemma~\ref{lemma:packing} links the packing constant of $(\Sigma,d_M)$ and of $(\Sigma,d_\Sigma)$.

\section{\bf Upper bound and measure concentration for $\sigma_k$}\label{section:GromovMilmanDiam}
The proof of Theorem~\ref{thm:upperboundDiam} depends on the min-max characterization of Steklov eigenvalues:
\begin{equation}
\label{eq:minmax}
\sigma_j = \min_{E \in \mathcal{H}_j} \max_{0 \neq u \in E} R_M(u),
\end{equation}
where $\mathcal{H}_j$ is the set of all $j+1$-dimensional subspaces in the Sobolev space $H^1(M)$, and
where
\begin{equation*}
\label{Rayleighquotient}
R_M(u) = \frac{\int_{M} \vert \nabla u \vert^2 dV_M}{\int_{\Sigma} u^2\, dV_{\Sigma}}
\end{equation*}
is the Rayleigh--Steklov quotient of $u$. 
We start with a simple bound which is adapted from the work of Gromov and Milman~\cite{GM}.
\begin{lemma}\label{lemma:GMnew}
  Let $A_i\subset\Sigma$ be disjoint measurable subsets, for $i=1,\cdots,k+1$, with positive measures $\mu_i:=|A_i|_\Sigma>0$. Suppose these subsets are quantitatively separated:
    $$\rho:=\frac{1}{2}\min_{i\neq j}d_M(A_i,A_j)>0.$$
    Then
    $$\sigma_k\leq\frac{|M|}{\rho^2\min\mu_i}.$$
\end{lemma}
  \begin{proof}
    We use standard trial functions $f_i$ that are supported in the pairwise-disjoint neighborhoods
    $A_i^\rho=\{x\in M\,:\,d_M(x,A_i)\leq \rho\}$ and have value 1 on $A_i$. These are defined by
    $$f_i(x)=
    \begin{cases}
      1-\frac{1}{\rho}d_M(x,A_i)&\text{ in }A_i^\rho,\\
      0&\text{ elsewhere }.
    \end{cases}$$
    Observe that
    $\|\nabla f_i\|^2=\frac{|A_i^\rho|_M}{\rho^2}$ and $\|f_i\|^2_{\Sigma}\geq \mu_i$ since $f_i\equiv 1$ on $A_i^\rho$, so that
    $$R(f_i)\leq \frac{|A_i^\rho|_M}{\rho^2\mu_i}\leq \frac{|M|}{\rho^2\min_i\mu_i}.$$
  \end{proof}
  \begin{rem}
    If one uses $A_i$ for $i=1,2,\cdots, 2k$ and suppose that $|A_1^\rho|_M\leq |A_2^\rho|_M\leq\cdots\leq |A_{2k+2}^\rho|_M$, then
    $|A_{k+1}^\rho|_M\leq|M|/k$ and one gets
    $$\sigma_k\leq\frac{|M|}{k\rho^2\min\mu_i}.$$
    This trick is often useful in improving the exponent on $k$ for bounds that are obtained using trial functions with  disjoint supports. This will be used in the proof of Theorem~\ref{thm:upperboundDiam} below.
  \end{rem}
  Lemma~\ref{lemma:GMnew} implies a concentration phenomena when $\sigma_1$ is large in comparison to the other constants involved.
  \begin{prop}\label{prop:concentrationNBD}
    Let $M$ be a compact manifold with boundary $\Sigma$. Let $A\subset\Sigma$ be a subset of positive measure $\mu=|A|_\Sigma>0$. Let $\rho>0$. If
    $\sigma_1\geq \frac{|M|}{\rho^2\mu}$, then
    $$|A^{2\rho}|_\Sigma\geq|\Sigma|-\frac{|M|}{\sigma_1\rho^2}.$$    
  \end{prop}
This is particularly interesting for families of manifolds $M_\eps$ such that $\sigma_1\to+\infty$, while $|\Sigma_\eps|$ and $|M_\eps|$ are independent of $\eps$. In that case, the extrinsic neighborhood $A^{2\rho}$ contains all of the boundary in the limit, however small the number $\rho$ is. This shows that the full boundary concentrates in the measure sence in the limit.
  \begin{proof}[Proof of Proposition~\ref{prop:concentrationNBD}]
    If $A^{2\rho}=\Sigma$, the statement is trivially true. Otherwise,
    define $B=\Sigma\setminus A^{2\rho}$, so that $d_M(A,B)=2\rho$ as suggested by the notation.
    It follows from Lemma~\ref{lemma:GMnew} that
$$\sigma_1\leq\frac{|M|}{\rho^2\min\{\mu,|\Sigma|-|A^{2\rho}|_\Sigma\}}.$$
Because $\sigma_1$ is large, that is $\sigma_1>\frac{|M|}{\rho^2\mu}$, one has 
$$\min\{\mu,|\Sigma|-|A^{2\rho}|_\Sigma\}=|\Sigma|-|A^{2\rho}|_\Sigma$$
so that
$$\sigma_1\leq\frac{|M|}{\rho^2(|\Sigma|-|A^{2\rho}|_\Sigma)}.$$
The proof is completed by reorganizing this inequality.
\end{proof}
  In order to prove Theorem~\ref{thm:upperboundDiam}, we will apply Lemma~\ref{lemma:GMnew} to well-chosen balls in the boundary component $\Sigma_j$.
    \begin{proof}[Proof of Theorem~\ref{thm:upperboundDiam}]
Let $\delta=\text{diam}_M(\Sigma_j)$. Consider $x_1,x_2,\cdots, x_{2k}\in\Sigma_j$ such that
  $$d_M(x_p,x_q)\geq\frac{\delta}{2k},\qquad \forall p\neq q.$$
  To see that this is possible, consider points $x_1,x_{2k}$ such that
  $$d_M(x_1,x_{2k})=\text{diam}_M(\Sigma_j)$$
  and use the concentric balls $B_M(x_1,\frac{p\delta}{2k})$ with $p=1,2,\cdots,2k$. Because $\Sigma_j$ is connected, it intersects each sphere
  $\partial B_M(x_1,\frac{i\delta}{2k})$. Any sequence of points $x_p\in \partial B_M(x_1,\frac{p\delta}{2k})$ will work.
  
  Now, use Lemma~\ref{lemma:GMnew} and its proof with $A_i=B_M(x_i,\frac{\delta}{8k})\cap\Sigma_i$, and observe that the triangle inequality gives $\rho\geq \delta/4k$, where $\rho$ is defined in Lemma~\ref{lemma:GMnew}. Hence, the Rayleigh quotients of the standard functions are controlled by
  $$R(f_i)\leq\frac{|A_i^\rho|_M}{\left(\frac{\delta}{4k}\right)^2|B_M(x_i,\frac{\delta}{8k})|_{\Sigma}}\leq
  \frac{16|M|}{\delta^2|B_M(x_i,\frac{\delta}{8k})|_{\Sigma}}k^2.$$
  Because the $2k$ sets $A_i^\rho$ are disjoint, we can reorder them to insure that
  $|A_i^\rho|_M\leq |M|/k$ for $i=1,2,\cdots,k+1$. This leads to the improved bound
  $$R(f_i)\leq
  \frac{16|M|}{\delta^2|B_M(x_i,\frac{\delta}{8k})|_{\Sigma}}k,\qquad\text{ for }i=1,2,\cdots,k+1.$$
  The main task is now to control the intrinsic volume $|B_M(x_i,\delta/8k)|_\Sigma$ from below.
  We split this in two cases:

\textbf{Case 1}: If
  $\delta/8k\geq \text{inj}(\Sigma_j)/4$ then, using that extrinsic balls are bigger than intrinsic balls of the same radius, Croke's inequality (\cite[Proposition 14]{Croke1980}) gives
  $$|B_M(x_i,\delta/8k)|_\Sigma\geq |B_\Sigma(x_i,\delta/8k)|_\Sigma\geq |B_\Sigma(x_i,\text{inj}(\Sigma_j)/4)|_\Sigma\geq c(n)\text{inj}(\Sigma_j)^n$$ and
  $$R(f_i)\leq c(n)\frac{|M|}{\delta^2\text{inj}(\Sigma_j)^n}k.
  $$

  \bigskip
  
\textbf{Case 2}: If $\delta/8k<\text{inj}(\Sigma_j)$ then Croke's inequality gives
  $$|B_M(x_i,\delta/8k)|_\Sigma\geq |B_\Sigma(x_i,\delta/8k)|_\Sigma\geq  c(n)\delta^n/k^n$$
  and
  $$R(f_i)\leq c(n)\frac{|M|}{\delta^{n+2}}k^{n+1}.$$

  \bigskip

Combining we get
  $$R(f_i)\leq c(n)\frac{|M|}{\text{diam}_M(\Sigma_j)^2}\times\frac{1}{\min\{\text{diam}_M(\Sigma_j)^n,\, \text{inj}(\Sigma_j)^n\}}k^{n+1}.$$
  The result now follows from the min-max characterisation of $\sigma_k$.
  \end{proof}

\section{\bf Upper bounds in term of distortion, packing and growth}\label{section:upperbounds}

The goal of this section is to prove Theorem~\ref{thm:upperbound}. We will prove the following slightly more general result.
\begin{theorem}\label{thm:upperboundGeneral}
  Let $M$ be a smooth connected compact Riemannian manifold of dimension $n+1$ with boundary $\Sigma$. Let $\Sigma_0$ be a connected component of the boundary.
Let $\G ,N_M$ be growth and extrinsic packing constants for $\Sigma_0$. Let $\Lambda$ be the distortion of $\Sigma_0$ in $M$. Then, the following holds for each $k\ge 1,$
  \begin{gather}\label{ineq:upperbound}
    \sigma_k\le 512N_M^3\G \Lambda^2\frac{\vert M\vert}{\vert \Sigma_0\vert^{\frac{n+2}{n}}}k^{2/n}.
  \end{gather}
\end{theorem}
Theorem~\ref{thm:upperbound} follows by taking $\Sigma_0$ to be the connected component of the boundary with the largest volume and observing that in this case
$|\Sigma_0|\geq|\Sigma|/b$.

The strategy is similar to the one used to prove~\cite[Theorem 1.1]{CGG}.
It is based on the following result, which is a simplification of~\cite[Lemma 2.1]{CEG2}.
\begin{lemma}\label{CMrevisited}
	Let $(X,d,\mu)$ be a complete, locally compact metric measure
	space, where $\mu$ is a non-atomic finite measure. Assume that for all
	$r>0$, there exists an integer $N$ such that each ball of radius
	$r$ can be covered by $N$ balls of radius $r/2$.
	Let $K>0$. If there
	exists a radius $r>0$ such that,
	for each $x \in X$
	$$
	\mu(B(x,r)) \le \frac{\mu(X)}{4N^2K},
	$$
	then, there exist $\mu$-measurable subsets $A_1,...,A_K$ of $X$
	such that, $\forall i\le K$, $\mu(A_i)\ge \frac{\mu(X)}{2NK}$
	and, for
	$i\not =j$, $d(A_i,A_j) \ge 3r$.
\end{lemma}
See~\cite[Lemma 4.1]{CGG} and the following paragraph for a discussion.
\begin{proof}[Proof of Theorem \ref{thm:upperboundGeneral}]
We apply Lemma~\ref{CMrevisited} to the metric measure space $(X,d,\mu)$ where $X=M$, and  $d=d_M$ is the extrinsic distance. The measure $\mu$ is  associated to the boundary component $\Sigma_0$: for a Borelian
subset $\mathcal{O}$ of $M$, we take 
$$\mu(\mathcal O)=|\Sigma_0 \cap \mathcal O|_\Sigma.$$
In particular, $\mu (M)$
is the usual volume $\vert \Sigma_0 \vert_\Sigma$ of this component. 
In order to estimate $\sigma_k$, we first construct $(2k+2)$ trial functions, so we will take $K=2k+2$. The constant $N$ in  Lemma~\ref{CMrevisited} is $N=N_M$ for $\Sigma_0$.
Let $\Lambda$ be the distortion of $\Sigma_0$ in $M$, so that $d_\Sigma(x,y)\leq \Lambda d_M(x,y)$, for each $x,y\in\Sigma_0$.  For each $x\in \Sigma_0$, this implies
$$
\{y\in \Sigma_0: d_M(y,x)\le r\} \subset \{y\in \Sigma_0: d_\Sigma(y,x)\le \Lambda r\}.
$$
In other words, $B^M(x,r)\cap\Sigma_0\subset B^\Sigma(x,\Lambda r)$. 
Recall that $\G $ is the growth constant of $\Sigma_0$. That is, for each $x\in\Sigma_0$ and each $r>0$,
$$
\mu(B^\Sigma(x,r)) \le \G r^n.
$$
This implies that for all $r>0$,
$$
\mu(B^M(x,r)) \le \G \Lambda^nr^n.
$$
Let $K=2(k+1)$.
Any
$r<\left(\frac{  \vert \Sigma_0\vert}{4N_M^2\Lambda^n\G K}\right)^{\frac{1}{n}}$ is such that 
$$
\mu(B^M(x,r))\le \frac{\vert \Sigma_0\vert}{ 4N_M^2 K}.
$$
It follows from Lemma~\ref{CMrevisited}
that there are $2(k+1)$ measurable subsets $A_i\subset \Sigma_0$, $i=1,...,2(k+1)$, that are  $3r$-separated for $d_M$ and satisfy
$$
\mu(A_i)\ge \frac{\vert \Sigma_0 \vert}{4N_M(2k+2)} \ge \frac{\vert \Sigma_0 \vert}{16Nk}.
$$
Taking
$$
r=\left( \frac{\vert \Sigma_0\vert }{C_2K}\right)^{\frac{1}{n}},
$$
with $C_2=8\G \Lambda^nN_M^2$ is enough to insure that
$$r<\left(\frac{\vert \Sigma_0\vert}{4N^2\G \Lambda^nK}\right)^{\frac{1}{n}}.$$
\medskip
This family of subset $A_i$ allows to construct a family of trial functions as follow. 
Let $A$ be one of the subsets $A_i$ and consider 
 $A^r:=\{x\in M: d_M(A,x) \le r \}$. There exists a function $f$ supported in $A^r$ whose Rayleigh-Steklov quotient satisfies
 $$R(f) =\frac{\int_M \vert \nabla f\vert^2}{\int_{\Sigma} f^2}\le \frac{1}{r^2}\frac{\volM{A^r}}{\volS{A}}.$$

\medskip
  For this, we consider the function
$$f(p)=\left\{
\begin{array}{cll}
	1 & {\rm if } & p\in A \\
	1-\frac{d_M(p,A)}{r} & {\rm if } & p\in \left(A^r\setminus A\right) \\
	0 & {\rm if } & p\in (A^r)^{c}  \ .
\end{array}\right.$$

It takes the value $1$ on $A$ and it gradient satisfies $\vert \nabla f\vert \le \frac{1}{r}$.
As we dispose from $2(k+1)$ subsets $A_i$, with $A_i^r$ and  $A_j^r$ disjoint if $i\not =j$, $k+1$ of them, say $A_1,...,A_{k+1}$ satisfy
$$
\vert A_i^r \vert \le  \frac{\vert M\vert}{k+1} \le  \frac{\vert M\vert}{k}.
$$
So, for the function $f$ associated to $A_i$, $i=1,...,k+1$, we have:
$$R(f)\le \frac{\vert M\vert}{k} \left(\frac{C_2k}{\vert \Sigma_0\vert}\right)^{\frac{2}{n}}\frac{8N_M k}{\vert \Sigma_0\vert}$$
Using that
$$C_2=8\G \Lambda^nN_M^2,$$
we obtain
\begin{align*}
  R(f)&\le
  8N\left(8\G \Lambda^nN_M^2\right)^{\frac{2}{n}}\frac{\vert M\vert}{\vert \Sigma_0\vert^{\frac{n+2}{n}}}k^{2/n}\\
  &=
  8^{1+2/n}N_M^\frac{n+4}{n}\G ^{2/n}\Lambda^2\frac{\vert M\vert}{\vert \Sigma_0\vert^{\frac{n+2}{n}}}k^{2/n}.
\end{align*}
Because $8^{1+2/n}\leq 8^3=512$, this completes the proof of Theorem~\ref{thm:upperboundGeneral}.
\end{proof}

\subsection{Importance of the geometric invariants}
\label{subsection:relevance}
In this last section, we discuss the effectiveness of Theorem~\ref{thm:upperbound} and Corollary~\ref{cor:upperboundmodif}: our goal is to explain why the various geometric constants play a meaningful role. To do this, we will exhibit various families of manifolds $M_\eps$ that satisfy $\sigma_1(M_\eps)\xrightarrow{\eps\to 0}+\infty$, while all but one of the geometric constant appearing in inequality~\ref{ineq:upperbound} or inequality~\ref{ineq:upperboundmodif} are independant of the parameter $\eps$.

We start by stating and proving the Lemma which was used to obtain Corollary~\ref{cor:upperboundmodif} from Theorem~\ref{thm:upperbound}
\begin{lemma}\label{lemma:packing}
    Let $\Lambda\geq 1$ be the distortion of the boundary\, $\Sigma$ in $M$. Let $\NS$ be a packing constant for $(\Sigma,d_\Sigma)$. Then $N_M=b\NS^{\log_2(2\Lambda)}$ is a packing constant for $(\Sigma,d_M)$.
\end{lemma}
\begin{proof}
  Let $\Sigma_1,\cdots,\Sigma_b$ be the connected components of $\Sigma$.
  Let $p\in\Sigma$ and $r>0$. Select one point $y_j\in B^M(p,r)\cap\Sigma_j$ whenever this intersection is nonempty. Then
  $$B^M(p,r)\cap\Sigma \subset \cup_jB^M(y_j,2r)\cap\Sigma_j\subset\cup_jB^\Sigma(y_j,2\Lambda r).$$
  For each $j$, there exists $\NS$ balls $B^\Sigma(x_i,\Lambda r)$ with centers $x_i\in\Sigma_j$, that cover $B^\Sigma(y_j,2\Lambda r)$. Each of these is covered by $\NS$ balls of radius
  $\frac{\Lambda r}{2}$, and repeating this process $m+1\in\N$ times leads to a cover of $B^\Sigma(y_j,2\Lambda r)$ by $\NS^{m+1}$ balls of radius $\frac{\Lambda r}{2^m}$. Now, for $m> 1+\log_2(\Lambda)$ the radius of the covering balls is smaller than $r/2$. 
  It follows that $B^M(p,r)\cap\Sigma$ is covered by at most $N:=b\NS^{m+1}$ balls of radius $r/2$:
  $$B^M(p,r)\cap\Sigma\subset\cup_{i=1}^{N}B^\Sigma(x_i,r/2)\subset\cup_{i=1}^{\NS}B^M(x_i,r/2).$$
\end{proof}

\subsubsection*{a) Importance of the distortion $\Lambda$}

In~\cite{CiG} the second author and D. Cianci constructed a family of Riemannian metrics $g_\eps$ on a manifold of $M$ dimension at least 4, such $\sigma_1\xrightarrow{\eps\to 0}+\infty$, while  $|M|$ is uniformly bounded and the restriction of $g_\eps$ to the boundary $\Sigma$ does not depend on $\eps$. If follows from~\eqref{ineq:upperboundmodif}
that the distortion $\Lambda$ must also become arbitrarily large as $\eps\to 0$.

\subsubsection*{b) Importance of the volume of $M$}   Given a compact manifold $M$ of dimension at least three, with boundary $\Sigma$, we constructed in~\cite{CEG3} a family of Riemannian metrics $g_\eps$ such that $\sigma_1\xrightarrow{\eps\to 0}+\infty$, while the distortion $\Lambda$ is uniformly bounded above and  the restriction of $g_\eps$ on the boundary $\Sigma$ does not depend on $\eps$. If follows from~\eqref{ineq:upperboundmodif}
that the volume $|M|$ must also be large.

\subsubsection*{c) Importance of the number of boundary components} In~\cite{ColbGirGraphSurface}, we constructed a sequence of compact surfaces $M_\ell$ with boundary such that $\sigma_1\xrightarrow{\ell\to\infty}+\infty$, and for each $\ell$:
  $$
  |\partial M_\ell|=|M_\ell|=\Lambda=1.
  $$
  Because the boundary is one-dimensional, the packing constant $\NS$ and the growth constant $\G $ are also independant of $\ell$.
  It follows from~\eqref{ineq:upperboundmodif} that the number of connected components $b$ of the boundary must satisfy $b\to\infty$. A similar construction works in arbitrary dimension.

  \subsubsection*{d) Importance of the volume of the boundary} Let $M$ be a closed Riemannian manifold, and consider the perforated domain $M_\eps:=M\setminus B(p,\eps)$. Then $b=1$ while $\NS$, $\Lambda$, $\G $ and $|M|$ are uniformy bounded. It is proved in~\cite{Brisson2022} that $\sigma_1\xrightarrow{\eps\to 0}+\infty$. 

  \subsubsection*{e) Importance of the packing and growth constants} 
  If Theorem~\ref{thm:upperbound} was true with $\G $ and $N_M$ removed, then Corollary~\ref{coro:gnycm} would hold without these constants appearing. That is, a universal upper bound on $\lambda_1(\Sigma)|\Sigma|^{2/n}$ would be provided for each compact manifold $\Sigma$. However it was proved by the first author and J. Dodziuk~\cite{ColboisDodziuk} that any closed manifold $\Sigma$ of dimension larger than 2 admits a Riemannian metric $g$ with arbititrarily large Laplace spectral gap $\lambda_1$.

\section*{Acknowledgements}
The authors would like to thank Iosif Polterovich and Jean Lagac\'e for useful comments on a preliminary version of this paper.
  The research of BC is supported by the Swiss National Science Foundation (SNF)
  (Grant 200021\_196894,  Geometric Spectral Theory) and (Grant 200020\_212570,  Geometric Spectral Theory).
  AG is supported by NSERC Discovery grant, Isoperimetry and spectral geometry RGPIN-2022-04247 and FRQNT (Projet de recherche en \'equipe: Applications de Dirichlet-Neumann: g\'eom\'etrie spectrale et probl\`emes inverses).

\bibliographystyle{plain}
\bibliography{biblioCG}

\end{document}